\documentclass[11pt]{amsart}
\usepackage[curve,color,graph]{xypic}
\usepackage{amscd,epsfig,amssymb,latexsym,amsmath,url,bm}
\usepackage{enumerate}
\usepackage{amsthm}		
\usepackage{mathrsfs}
\usepackage{array}
\usepackage[all]{xy}
\usepackage{graphicx}
\usepackage{color}

\newtheorem{theorem}{Theorem}
\newtheorem{proposition}{Proposition}
\newtheorem{lemma}{Lemma}
\newtheorem{corollary}{Corollary}

\newtheorem*{theoremb}{Theorem [Boyd-Schulz]}
\newtheorem*{theoremb2}{Theorem [Boyd-Schulz]}

\title{Geometric Limits of Julia Sets \\with Parameters on the Circle}
\begin{author}[S.R. Kaschner]{Scott R. Kaschner}
\email{skaschner@math.arizona.edu}
\address{ %
University of Arizona Department of Mathematics\\
Mathematics Building, Room 317\\
617 N. Santa Rita Ave.\\
Tucson, AZ 85721-0089\\
 United States }
\end{author}
\begin{author}[R. Romero]{Reaper Romero}
\email{rromero@email.arizona.edu}
\address{ %
University of Arizona Department of Mathematics\\
Mathematics Building\\
617 N. Santa Rita Ave.\\
Tucson, AZ 85721-0089\\
 United States }
\end{author}
\begin{author}[D. Simmons]{David Simmons}
\email{davidsimmons@email.arizona.edu}
\address{ %
University of Arizona Department of Mathematics\\
Mathematics Building\\
617 N. Santa Rita Ave.\\
Tucson, AZ 85721-0089\\
 United States }
\end{author}

\begin{document}


\begin{abstract}
We show that the geometric limit as $n \rightarrow \infty$ of the filled Julia sets $K(P_{n,c})$ for the maps $P_{n,c}(z) = z^n + c$ does not exist for almost every $c$ on the unit circle.  Furthermore, we show that there is always a subsequence along which the limit does exist and equals the unit circle, and this is used to show that for certain parameters, the geometric limit of the Julia sets $J(P_{n,c})$ is the unit circle.
\end{abstract}
\maketitle

\setcounter{section}{1}

Consider the family of maps
\[P_{n,c}(z)\ =\ z^n+c,\]
where $n\geq2$ is an integer and $c\in\mathbb C$ is a parameter.  These maps all share the quality that there is only one free critical point; that is, the critical point at infinity is fixed under iteration, and the iterates of the remaining critical point, $z=0$, depend on both $c$ and $n$.  Because of this uni-critical property, many dynamical properties of the classical quadratic family $z\mapsto z^2+c$ are also exhibited by this family of maps. Details of this family are readily available in the literature \cite{mcmullen, SCHLEICHER, LYUBICH}.

In this note, we will consider the filled Julia set $K(P_{n,c})$, the set of points in $\mathbb C$ that remain bounded under iteration and its boundary, the Julia set $J(P_{n,c})$.  In \cite{boyd}, the structure of the filled Julia set $K(P_{n,c})$ and its boundary $J(P_{n,c})$, the Julia set, as $n\rightarrow\infty$ was examined.  One of the major results is this work was
\begin{theoremb}
Let $c\in\mathbb C$, and let $CS(\hat{\mathbb C})$ denote the collection of all compact subsets of $\hat{\mathbb C}$.  Then under the Hausdorff metric $d_{\mathcal H}$ in $CS(\hat{\mathbb C})$,
\begin{itemize}
\item[(1)] If $c\in\mathbb C\backslash\overline{\mathbb D}$, then
\[\displaystyle\lim_{n\rightarrow\infty}J(P_{n,c})=\displaystyle\lim_{n\rightarrow\infty}K(P_{n,c})=S^1.\]
\item[(2)] If $c\in\mathbb D$, then
\[\displaystyle\lim_{n\rightarrow\infty}J(P_{n,c})=S^1\mbox{ and }\displaystyle\lim_{n\rightarrow\infty}K(P_{n,c})=\overline{\mathbb D}.\]
\item[(3)] If $c\in S^1$, then if
$\displaystyle\lim_{n\rightarrow\infty}J(P_{n,c})$ and/or $\displaystyle\lim_{n\rightarrow\infty}K(P_{n,c})$ (and/or any liminf or limsup) exists, it is contained in $\overline{\mathbb D}$.
\end{itemize}
\end{theoremb}

The purpose of this note is to improve part (3) of this result.  While there may be no limit as $n\rightarrow\infty$ for $J(P_{n,c})$ or $K(P_{n,c})$, experimentation suggests given $c\in S^1$, there is almost always a predictable pattern for the filled Julia set for $P_{n,c}$ as $n\rightarrow\infty$.  This experimentation led to the following result:

\begin{theorem}\label{main} Let $c=e^{2\pi i\theta}\in S^1$.  Suppose $\theta\neq0$ and for all $p\in\mathbb N$ and $q\in\mathbb Z$, $\theta\neq\frac{3q\pm1}{3(6p-1)}$. Then 
\[\displaystyle\lim_{n\rightarrow\infty}K(P_{n,c})\]
does not exist.  Moreover, if we also suppose $\theta$ is rational, 
then there exist subsequences $a_k$ and $b_k$ partitioning $\mathbb N$ such that
\[\displaystyle\lim_{k\rightarrow\infty}K(P_{a_k,c})\ =\ S^1\mbox{\hspace{.25in}and\hspace{.2in}}
\displaystyle\lim_{k\rightarrow\infty}K(P_{b_k,c})\ =\ \overline{\mathbb D}.\]
\end{theorem}

\begin{corollary} Let $c=e^{2\pi i\theta}$, where $\theta\neq0$ is rational and for all $p\in\mathbb N$ and $q\in\mathbb Z$, $\theta\neq\frac{3q\pm1}{3(6p-1)}$. Then 
\[\displaystyle\lim_{n\rightarrow\infty}J(P_{n,c})=S^1.\]
\end{corollary}

In Section \ref{BG}, we present the background material and motivation for this result.  The proof of Theorem \ref{main} is the focus of Section \ref{PROOF}.

The authors are grateful to both Mikhail Stepenov at the University of Arizona and Signe Emalia Jensen at Northwestern University for their helpful suggestions.

\section{Background and Motivation}\label{BG}

\subsection{Notation and Terminology}

The main results in this note rely on the convergence of sets in $\hat{\mathbb C}$, where the convergence is with respect to the Hausdorff metric.  Given two sets $A,B$ in a metric space $(X,d)$, the Hausdorff distance $d_{\mathcal H}(A,B)$ between the sets is defined as
\begin{eqnarray*}
d_{\mathcal H}(A,B)&=&\max\left\{\sup_{a\in A}d(a,B),\sup_{b\in B}d(b,A)\right\}\\
&=&\max\left\{\sup_{a\in A}\inf_{b\in B}d(a,b),\sup_{b\in B}\inf_{a\in A}d(a,b)\right\}.
\end{eqnarray*}
Each point in A has a minimal distance to B, and vice versa.  The Hausdorff distance is the maximum of all these distances.  For example, a regular hexagon $A$ inscribed in a circle $B$ of radius $r$ has sides of length $r$.  In this case, $d_{\mathcal A}(A,B)=r(1-\sqrt3/2)$, the shortest distance from the circle to the midpoint of a side of the hexagon.  See Figure \ref{circle}.  Julia sets $J(P_{n,c})$ and filled Julia sets $K(P_{n,c})$ are compact \cite{beardon} in the compact space $\hat{\mathbb C}$.  Moreover, with the Hausdorff metric $d_{\mathcal H}$, $\hat{\mathbb C}$ is complete \cite{henrikson}.

Suppose $S_n$ and $S$ are compact subsets of $\mathbb C$.  If for all $\epsilon>0$, there is $N>0$ such that for any $n\geq N$, we have $d_{\mathcal H}(S_n, S)<\epsilon$, then we say $S_n$ converges to $S$ and write $\lim_{n\rightarrow\infty}S_n=S$.

We adopt the notation from \cite{boyd}.  For an open annulus with radii $0<r<R$, 
\begin{equation*}
\mathbb A(r,R):=\{z\in\mathbb C\colon r<|z|<R\}.
\end{equation*}
Also, the open ball of radius $\epsilon>0$ centered at $z$ will be denoted $B(z,\epsilon)$.

\subsection{Motivation}

A basic fact from complex dynamics (see \cite{beardon} or \cite{milnor}) is that $K(P_{n,c})$ is connected if and only if the orbit of $0$ stays bounded; otherwise it is totally disconnected.   For each $n\geq2$, we define the Multibrot sets
\begin{equation*}
\mathcal M_n:=\{c\in\mathbb C\colon J(P_{n,c})\mbox{ is connected}\}.
\end{equation*}
Since $0$ is the only free critical point, $\mathcal M_n$ is also the set of parameters $c$ such that the orbit of $0$ under iteration by $P_{n,c}$ remains bounded \cite{milnor}.  Since the maps $P_{n,c}$ are uncritical, much of their dynamical behavior mimics the family of complex quadratic polynomials \cite{SCHLEICHER}.


It was proven in \cite{boyd} that for sufficiently large $N$, 
\begin{itemize}
\item[(1)] $c\in\mathbb D$ implies for any $n\geq N$, $0\in K(P_{n,c})$ (the orbit of 0 is bounded and $c\in\mathcal M_n$), and
\item[(2)] $c\in\mathbb C\backslash\overline{\mathbb D}$ implies for any $n\geq N$, $0\notin K(P_{n,c})$ (the orbit of 0 is not bounded and $c\notin\mathcal M_n$).
\end{itemize}
For parameters $c\in S^1$, $P_{n,c}(0)\in S^1$ for any $n$, and this obstructs the direct proof that the orbit of 0 remains bounded (or not).  However, one finds that in most cases, $P_{n,c}^2(0)\notin S^1$ and should expect that in these situations, determining whether the orbit of zero stays bounded depends heavily on where $P_{n,c}^2(0)$ is relative to the circle.  In fact, working with the second iterate of 0 will be sufficient for all of our proofs. 

\begin{figure}
\centering
\scalebox{.8}{\includegraphics{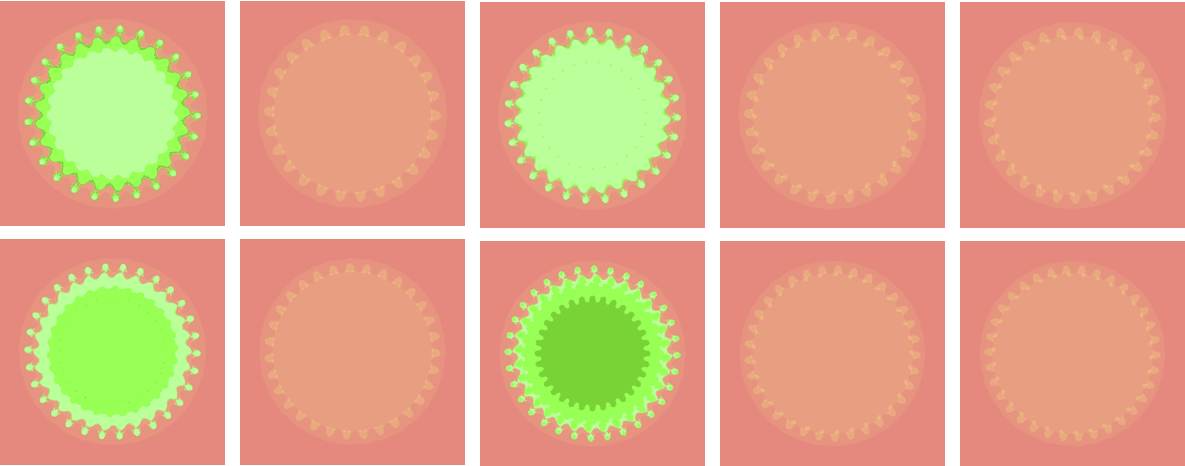}}
\caption{$J(P_{n,c})$ for $c=e^{4\pi i/5}$ and $n=25\dots34$, starting from the upper left to the lower right.}\label{Julias}
\end{figure}

Noting that $P_{n,c}^2(0)=P_{n,c}(c)$, we have the following convenient formula:
\begin{proposition}\label{formula}
For $c=e^{2\pi i\theta}\in S^1$ and any positive integer $n$, $|P_{n,c}(c)|\geq1$ if and only if
\[\cos(2\pi\theta(n-1))\geq-\frac{1}{2},\]
where equality holds if and only if $|P_{n,c}(c)|=1$.
\end{proposition}

\begin{proof}
Note first that for $c=e^{2\pi i\theta}$, we have $P_{n,c}(c)=(e^{2\pi i \theta})^{n}+e^{2\pi i \theta}$, so
\begin{eqnarray*}
P_{n,c}(c)&=&\cos(2\pi\theta n)+i \sin(2\pi\theta n)+ \cos(2\pi\theta) + i \sin(2\pi\theta)\\
&=&\cos(2\pi\theta n)+\cos(2\pi\theta)+i( \sin(2\pi\theta n)+ \sin(2\pi\theta)).
\end{eqnarray*}	
If $P_{n,c}(c)\geq1$, then
\begin{eqnarray*}
1&\leq&(\cos(2\pi\theta n)+\cos(2\pi\theta))^{2}+( \sin(2\pi\theta n)+ \sin(2\pi\theta))^{2}\\
&=&2\cos(2\pi\theta n)\cos(2\pi\theta)+2\sin(2\pi\theta n)\sin(2\pi\theta)+2\\
&=&2\cos(2\pi\theta(n-1))+2
\end{eqnarray*}
from which the result follows.
\end{proof}

Experimentation indicates that $P_{n,c}(c)$ being inside (or outside) $S^1$ very consistently dictates that $c\in\mathcal M_n$ (or $c\notin\mathcal M_n$).  See Figure \ref{Julias}.  Then the condition on $P_{n,c}(c)$ from Proposition \ref{formula} can be used to very consistently predict the structure of $K(P_{n,c})$, which Proposition \ref{formula} also suggests is periodic with respect to $n$. This will be made precise (with quantifiers) in Proposition \ref{trap} below.

More efficient experimentation with checking whether the orbit of 0 stays bounded clearly present this periodic (with respect to $n$) structure for $K(P_{n,c})$ when $c$ is a rational angle on $S^1$. 
Figure \ref{Stars} shows powers $421\leq n\leq450$ and $c=e^{\pi ip/q}\in S^1$ where $q=15$ and $p$ is an integer with $1\leq p\leq 30$.  A star indicates the Julia set $J(P_{n,c})$ is connected.  There is, however, an inconsistency when the orbit of 0 remains on $S^1$. Note that the situation in which $P_{n,c}(c)\in S^1$ corresponds to having $\cos(2\pi\theta(n-1))=-1/2$.  This can be seen in Figure \ref{Stars} for $n=426$ and $2\theta=26/15$ and $2\theta=28/15$.  The program that generated this data can provide a similar table for any equally distributed set of angles and any consecutive set of iterates.

\begin{figure}
\centering
\scalebox{.8}{\includegraphics{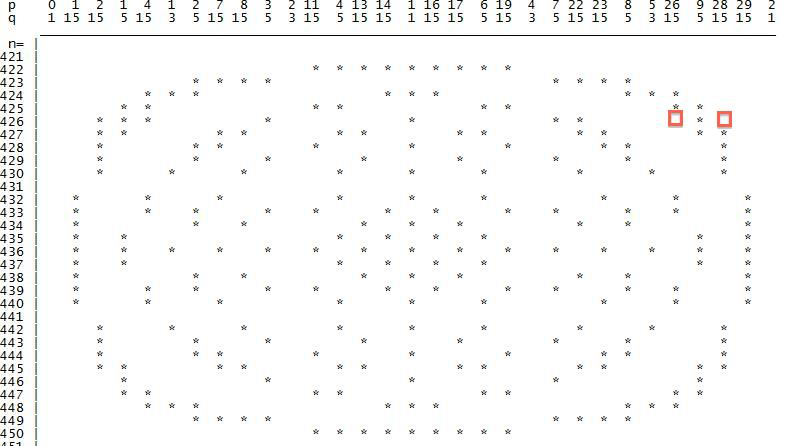}}
\caption{A star indicated $J(P_{n,c})$ is connected, where $c=e^{\pi ip/q}$}
\label{Stars}
\end{figure}

This experimentation yields an intuition that is supported further by another result from \cite{boyd}:
\begin{theoremb2}
Under the Hausdorff metric $d_{\mathcal H}$ in $CS(\hat{\mathbb C})$,
\begin{equation*}
\lim_{n\rightarrow\infty}M(P_{n,c})=\overline{\mathbb D}.
\end{equation*}
\end{theoremb2}
For a fixed $c\in S^1$, as $n$ increases, $c$ will fall into and out of $\mathcal M_n$.  See Figure \ref{Mandelbrot}.  Thus, Proposition \ref{formula} provides nice visual evidence that this is truly periodic behavior.  The Multibrot sets in Figure \ref{Mandelbrot} are in logarithmic coordinates, so the horizontal axis is the real values $-1\leq\theta\leq1$, where $c=e^{2\pi i\theta}$. We are using logarithmic coords since we are interested in the angle $\theta$.

\begin{figure}
\centering
\scalebox{.66}{\includegraphics{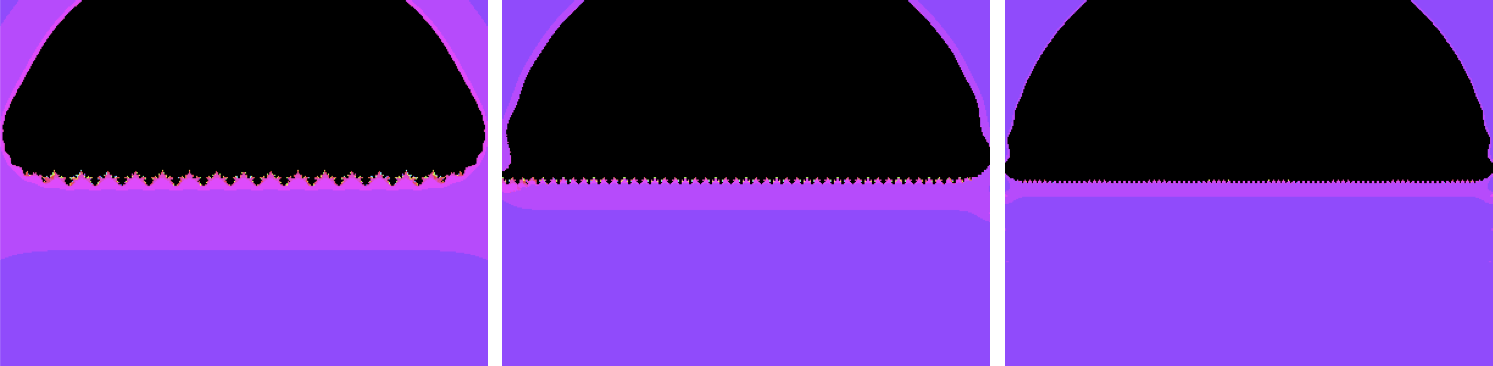}}
\caption{$\mathcal M_n$, where $c=e^{2\pi i\theta}$, $\theta\in\mathbb C$, and $n=10,25,50$.  Almost all fixed Re$\theta$, falls into and out of $\mathcal M_n$ as $n$ increases.}
\label{Mandelbrot}
\end{figure}

It remains an open question what happens for parameters with angles $\theta=\frac{3q\pm1}{3(6p-1)}$ for $p\in\mathbb N$ and $q\in\mathbb Z$.  We prove in Proposition \ref{fixed} that the parameters corresponding to these angles force $P_{n,c}(c)$ to be a fixed point on $S^1$.  In this case, the critical orbit is clearly bounded, so we know the filled Julia set $K(P_{n,c})$ must be connected.  See Figure \ref{boundary}.  However, the behavior of the boundary $J(P_{n,c})$ is extremely complicated, as in the left-most image in Figure \ref{boundary}.

\begin{figure}
\centering
\scalebox{.27}{\includegraphics{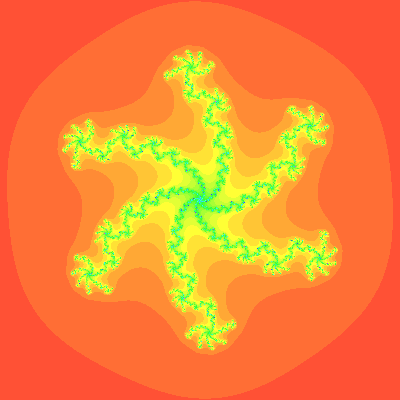}}
\scalebox{.27}{\includegraphics{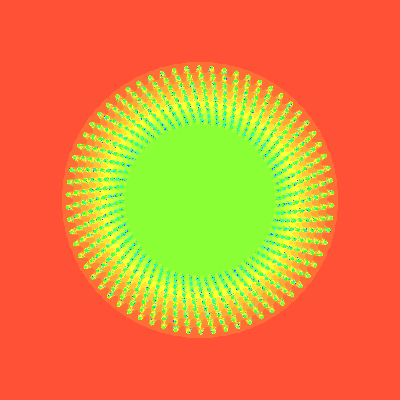}}
\scalebox{.27}{\includegraphics{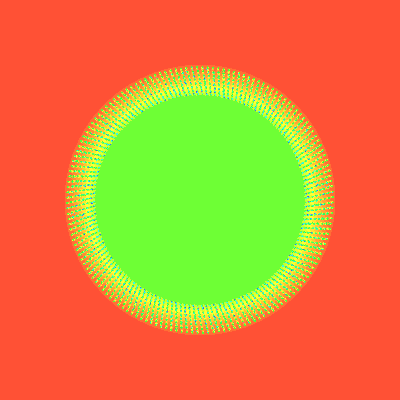}}
\scalebox{.31}{\includegraphics{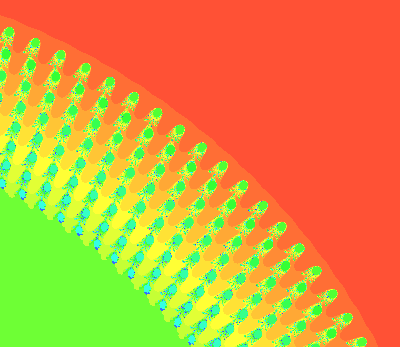}}
\caption{From left to right: $K(P_{n,c})$ for $c=e^{2\pi i/15}$ and $n=6,66,156$.  The far left image is a closer look at the boundary when $n=165$}
\label{boundary}
\end{figure}

\pagebreak

\section{Proof of Theorem \ref{main}}\label{PROOF}

We now prove that $P_{n,c}(c)\notin S^1$ does allow us to determine whether $c\in \mathcal M_n$.

\begin{proposition}\label{trap}
Let $c\in S^1$. For any $\epsilon > 0$ there exists $N > 0$ so that for all $n\geq N$ one has:
\begin{itemize}
\item[1.] if $|P_{n,c}(c)|<1-\epsilon$, then $\mathbb D_{1-\epsilon} \subset K(P_{n,c})$.
\item[2.] if $|P_{n,c}(c)|>1+\epsilon$, then $\mathbb D_{1-\epsilon} \subset\mathbb C \setminus K(P_{n,c})$.
\end{itemize}
\end{proposition}

Noting that $0\in\mathbb D_{1-\epsilon}$, it follows immediately from Propositions \ref{formula} and \ref{trap} that the orbit of 0 is bounded (or not) depending respectively on whether $P_{n,c}(c)$ is inside $\mathbb D_{1-\epsilon}$ (or outside $\mathbb D_{1+\epsilon}$).  That is,


\begin{corollary}\label{trap2}
For all $\epsilon>0$, there is an $N$ such that for any $n\geq N$,
\begin{itemize}
\item[1.] if $\cos(2\pi \theta (n-1)) < -1/2-\epsilon/2$, then $K(P_{n,c})$ is connected and
\item[2.] if $\cos(2\pi \theta(n-1)) > -1/2+\epsilon/2$, then $K(P_{n,c})$ is totally disconnected and $K(P_{n,c})=J(P_{n,c})$.
\end{itemize}
\end{corollary}


\begin{proof}[Proof of Proposition \ref{trap}]
Fix $c\in S^1$. Let $\epsilon>0$ and $r_n:=|P^2_{n,c}(0)|=|c^n+c|$.  Observe
\begin{eqnarray*}
\left|P^2_{n,c}(z)\right|\ =\ |(z^n + c)^n + c|&=&\left|c^n + c + \sum_{k=1}^n \binom{n}{k} (z^n)^k c^{n-k}\right|\\
&\leq& \left|c^n + c\right| + \sum_{k=1}^n \binom{n}{k} |z|^{n k}\ =\ r_n + (1+|z|^n)^n - 1.
\end{eqnarray*} 
Then $|P^2_{n,c}(z)|\leq|z|$ when $ r_n+(1+|z|^n)^n-1<|z|$.  That is, for any $\eta\in(0,1)$, if
\begin{eqnarray}\label{inequality1}
r_n&\leq&\eta+1-(1+\eta^n)^n,
\end{eqnarray}
then the disk $\mathbb D_{\eta}$ is forward invariant under $P^2_{n,c}$. Note that $(1+\eta^n)^n>1$ and for fixed $\eta$, $(1+\eta^n)^n\rightarrow1$ as $n\rightarrow\infty$.  Fix $\eta=1-\epsilon/2$, so there is a positive integer $N$ such that for all $n\geq N$,
\begin{equation*}
(1+\eta^n)^n-1<\frac{\epsilon}{2}.
\end{equation*}
Thus, for any $n\geq N$ such that $r_n<1-\epsilon$, 
\begin{equation*}
r_n<\eta-\frac{\epsilon}{2}<\eta+1-(1+\eta^n)^n,
\end{equation*}
so, $\mathbb D_{1-\epsilon}\subset\mathbb D_{\eta}$ is forward invariant under $P^2_{n,c}$.  This implies that the orbit of any point in $\mathbb D_{1-\epsilon}$ must be bounded in a disk of radius $\eta^n+1$, so we have $\mathbb D_{1-\epsilon}\subset K(P_{n,c})$.

On the other hand, note that
\begin{eqnarray*}
\left|P^2_{n,c}(z)\right|\ =\ |(z^n + c)^n + c|
&\geq& \left|\left|c^n + c\right| - \sum_{k=1}^n \binom{n}{k} |z|^{n k}\right|\ =\ \left|r_n - (1+|z|^n)^n + 1\right|.
\end{eqnarray*} 
Again, fix $\eta=1-\epsilon/2$, so there is an $N$ such that for any $n\geq N$, if $r_n>1+\epsilon$ and $|z|<1-\epsilon/2$, then
\begin{equation*}
(1+|z|^n)^n - 1<(1+\eta^n)^n-1<\frac{\epsilon}{2}.  
\end{equation*}
That is, for $n\geq N$ and $z\in\mathbb D_{\eta}$,
\begin{equation*}
|P^2_{n,c}(z)|\geq |r_n - (1+|z|^n)^n + 1|\geq1+\frac{\epsilon}{2}.
\end{equation*}
By Lemma \ref{bound}, we can also choose $N$ large enough that $K(P_{n,c})\subset\mathbb D_{1+\epsilon/2}$ as well.  Then for any $n>N$ and $z\in\mathbb D_{\eta}$, if $|P_{n,c}(c)|=r_n<1+\epsilon$, then $P^2_{n,c}(z)\notin K(P_{n,c})$.  It follows that $z\notin K(P_{n,c})$, so $\mathbb D_{\eta}\subset\mathbb C\backslash K(P_{n,c})$. 
\end{proof}

What remains is to examine $c\in S^1$ such that $P_{n,c}(c)\in S^1$ as well.  This case is simpler and occurs less frequently than one might expect.

\begin{proposition}\label{fixed}
Let $c=e^{2\pi i\theta}$ and $P_{n,c}(z)=z^n+c$.  
Then $P_{n,c}^2(c)\in S^1$ if and only if $P_{nc}(c)$ is a fixed point, in which case, $(n,\theta)\in N$, where
\[N:=\left\{(n,\theta)\in\mathbb N\times\mathbb R\mid n=6p, \theta=\frac{3q\pm1}{3(6p-1)},\mbox{ where }p\in\mathbb N\mbox{ and }q\in\mathbb Z\right\}.\]
\end{proposition}

\begin{proof}
Since $|c|=1$, note that the set $S^1-c:=\{z-c\mid z\in S^1\}$ is a circle centered at $-c\in S$, so it intersects $S^1$ in exactly two points, call them $a_0$ and $b_0$.  By construction, $a_0+c, b_0+c\in S^1$, so define
\begin{eqnarray*}
a&:=&a_0+c\\
b&:=&b_0+c.
\end{eqnarray*}
Moreover, the points $\{c, a, a_0, -c, b_0, b\}$ form a hexagon inscribed in $S^1$ whose sides are all length one.  Thus, we have
\begin{eqnarray*}
a&=&e^{2\pi i(\theta+1/6)}\\
a_0&=&e^{2\pi i(\theta+1/3)}\\
b_0&=&e^{2\pi i(\theta-1/3)}\\
b&=&e^{2\pi i(\theta-1/6)}.
\end{eqnarray*}

\begin{figure}
\label{circle}
\scalebox{.5}{\includegraphics{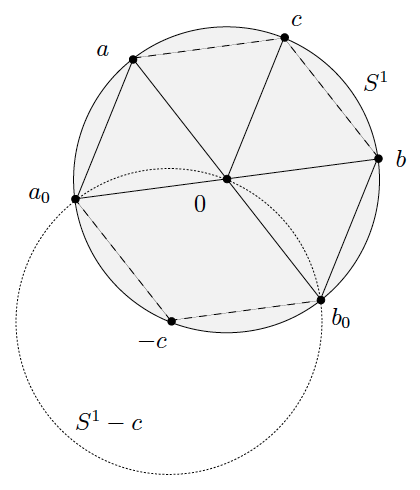}}
\caption{$P_{n,c}(c)$ is on the circle if and only if $c^n=a_0$ or $a^n=b_0$.}
\end{figure}

See Figure \ref{circle}.  For any $z\in S^1$, we have that $P_{n,c}(z)=z^n+c$ and $z^n\in S^1$, so $P_{n,c}(z)\in S^1$ if and only if
\begin{equation*}
z^n\in(S^1-c)\cap S^1\ =\ \{a_0,b_0\};
\end{equation*}
that is, $P_{n,c}(z)\in\{a,b\}$.  It follows that $|P_{n,c}^k(c)|=1$ for all $k\geq0$ if and only if one of the following is true: $a$ is a fixed point, $b$ is a fixed point, or $a$ and $b$ are a two-cycle.

Assume that $P_{n,c}(c)\in S^1$.  First observe that $P_{n,c}(c)\in\{a, b\}$, so
\begin{eqnarray*}
P_{n,c}(c)=e^{2\pi i(\theta\pm1/6)}.
\end{eqnarray*}
Since $P_{n,c}(c)=c^n+c=e^{2\pi i\theta n}+e^{2\pi i\theta}$, it follows that
\begin{eqnarray*}
e^{2\pi i\theta n}=e^{2\pi i(\theta\pm1/6)}-e^{2\pi i\theta}=e^{2\pi i(\theta\pm1/3)}.
\end{eqnarray*}
Thus, $\theta n=\theta\pm1/3+q$ for some integer $q$, so
\begin{eqnarray}\label{firstit}
\theta(n-1)&=&q+\frac{1}{3}\mbox{ if }P_{n,c}(c)=a\mbox{ and}\label{firstitp}\\
\theta(n-1)&=&q-\frac{1}{3}\mbox{ if }P_{n,c}(c)=b\label{firstitm}.
\end{eqnarray}
Proceeding to the next iterate, note that $P_{n,c}^2(c)\in\{a, b\}$ as well, so we need only examine $P_{n,c}(a)$ and $P_{n,c}(b)$.  Since $P_{n,c}(a), P_{n,c}(b)\in\{a,b\}$, it must be for some integer $p_0$,
\begin{eqnarray*}
P_{n,c}\left(e^{2\pi i(\theta\pm1/6)}\right)=e^{2\pi i(\theta\pm1/6)n}+e^{2\pi i\theta}
\in\{a,b\}=\left\{e^{2\pi i(\theta+1/6+p_0)},e^{2\pi i(\theta-1/6+p_0)}\right\}.
\end{eqnarray*}
Then it follows that from the definition of $a$ and $b$ that $e^{2\pi i(\theta\pm1/6+p_0)}\in\{a_0,b_0\}$, so we have  $(\theta\pm1/6)n=\theta\pm1/3+p_0$.
In particular, 
\begin{eqnarray}
(n-1)\theta&=&p_0+\frac{1}{3}-\frac{n}{6},\mbox{ \ if }P_{n,c}(a)=a,\label{atoa}\\
(n-1)\theta&=&p_0-\frac{1}{3}-\frac{n}{6},\mbox{ \ if }P_{n,c}(a)=b,\label{atob}\\
(n-1)\theta&=&p_0+\frac{1}{3}+\frac{n}{6},\mbox{ \ if }P_{n,c}(b)=a,\mbox{ and}\label{btoa}\\
(n-1)\theta&=&p_0-\frac{1}{3}+\frac{n}{6},\mbox{ \ if }P_{n,c}(b)=b.\label{btob}
\end{eqnarray}
If $a$ and $b$ are a two cycle, then equations (\ref{atob}) and (\ref{btoa}) together imply $q\pm1/3=p_0$.  This contradicts the fact that $q$ and $p_0$ are both integers.  A similar contradiction arises from the cases when $P_{n,c}(b)=a$ and $a$ is fixed, or when $P_{n,c}(a)=b$ and $b$ is fixed.

The only remaining possibilities are that $P_{n,c}(c)=P_{n.c}(a)=a$ or $P_{n,c}(c)=P_{n.c}(b)=b$.  Thus, we have shown that $|P_{n,c}^k(c)|=1$ for all $k\geq0$ if and only if for all $k\geq1$, $P_{n,c}^k(c)=a$ or $P_{n,c}^k(c)=b$.

It remains to show that $(n,\theta)\in N$ is an equivalent statement.  Supposing that for all $k\geq1$,  $P_{n,c}^k(c)=a$ or $P_{n,c}^k(c)=b$, we have
\begin{equation*}
q\pm\frac{1}{3}=\theta(n-1)=p_0\pm\frac{1}{3}\mp\frac{n}{6}.
\end{equation*}
From this equation, one can see that $n=6p$, where $p=q-p_0\in\mathbb N$.  Moreover, the equations (\ref{firstitp}) and (\ref{firstitm}) derived from the first iterate of $c$ yield
\begin{equation*}
\theta(n-1)=q\pm\frac{1}{3},
\end{equation*}
so
\begin{equation*}
\theta=\frac{3q\pm1}{3(n-1)}=\frac{3q\pm1}{3(6p-1)}.
\end{equation*}
\end{proof}

The following lemmas are from \cite{boyd}.  The third is a subtle variation, so we include the proof.

\begin{lemma}[Boyd-Schulz]\label{bound}
Let $c\in\mathbb C$.  For any $\epsilon>0$, there is an $N$ such that for all $n\geq N$,
\[K(P_{c,n})\subset\mathbb D_{1+\epsilon}.\]
\end{lemma}

\begin{lemma}[Boyd-Schulz]
Let $z\in J(P_{n,c})$.  If $\omega$ is an $n$-th root of unity, then $\omega z\in J(P_{n,c})$.
\end{lemma}

\begin{lemma}[Boyd-Schulz]\label{ball}
Let $\epsilon>0$ and $c=e^{2\pi i\theta}\in S^1$ such that $\theta\neq\frac{3q\pm1}{3(6p-1)}$ for any $p\in\mathbb N$ and $q\in\mathbb Z$.  There is an $N\geq2$ such that for all $n\geq N$ and for any $e^{i\phi}\in S^1$,
\[B(e^{i\phi},\epsilon)\cap J(P_{n,c})\neq\emptyset.\]
\end{lemma}
\begin{proof}
By Proposition \ref{trap}, there is an $N_1$ such that for any $n\geq N_1$, we have $J(P_{n,c})\subset\mathbb A(1-\epsilon/2,1+\epsilon/2)$.  Let $e^{i\phi}\in S^1$ and $\alpha>0$ be the angle so that
\[U:=\{re^{i\tau}\colon r>0,\phi-\alpha<\tau<\phi+\alpha\}\cap\mathbb A(1-\epsilon/2,1+\epsilon/2\}\]
is contained in $B(e^{i\phi},\epsilon)$.  The same $\alpha$ works for each different $\phi$.

For any $n$, let $\omega_n=e^{2\pi i/n}$, and choose $N>N_1$ such that $2\pi/N<\alpha$, noting that $N$ is also independent of $\phi$.  We have $2\pi/n<\alpha$ for any $n\geq N$.

Since $J(P_{n,c})$ is nonempty for any $n$ \cite{milnor}, choose $z_n\in J(P_{n,c})$ for each $n\geq N$. Then for some integer $1\leq j_n\leq n-1$, we have
\[\omega_n^{j_n}z_n\in U\subset B(e^{i\phi},\epsilon).\]
Thus, for all $n\geq N$, $B(e^{i\phi},\epsilon)\cap J(P_{n,c})\neq\emptyset$.
\end{proof}

\begin{proof}[Proof of Theorem \ref{main}]

Fix $c=e^{2\pi i\theta}\in S^1$ and assume $\theta\neq\frac{3q\pm1}{3(6p-1)}$ for any $p\in\mathbb N$ and $q\in\mathbb Z$.  Then by Proposition \ref{fixed}, $|P_{n,c}(c)|\neq1$, and by Proposition \ref{formula}, we have $\cos(2\pi\theta(n-1))\neq-\frac{1}{2}$.  In particular, 
\begin{enumerate}
\item $|P_{n,c}(c)|<1$ when $\cos(2\pi\theta(n-1))<-\frac{1}{2}$, and
\item $|P_{n,c}(c)|>1$ when $\cos(2\pi\theta(n-1))>-\frac{1}{2}$.
\end{enumerate}
Note that $\cos(2\pi\theta(n-1))$ has period $1/\theta$ as a function of $n$.  If $\theta$ is a rational number, then this function takes a finite number of values.  In this case, $|P_{n,c}(c)|$ can be bound away from $S^1$ by a fixed distance for any $n$. Let $\epsilon > 0$ be smaller than this minimum distance.   Then, Proposition \ref{trap} gives that that there is $N > 0$ such that for all $n\geq N$, we have either 
\begin{itemize}
\item[1.] $|P_{n,c}(c)|<1-\epsilon$ and $\mathbb D_{1-\epsilon} \subset K(P_{n,c})$, or
\item[2.] $|P_{n,c}(c)|>1+\epsilon$ and $\mathbb D_{1-\epsilon} \subset\mathbb C \setminus K(P_{n,c})$.
\end{itemize}
Moreover, if we consider $\theta$ as a rational rotation of the circle, the periodic orbit (with respect to $n$) induces intervals on $S^1$ that are permuted by this rotation \cite{KATOK}.  Since $\cos(2\pi\theta(n-1))\neq-\frac{1}{2}$, we must have $n$ and $m$ such that $\cos(2\pi\theta(n-1))\geq-\frac{1}{2}$ and $\cos(2\pi\theta(m-1))\geq-\frac{1}{2}$.  Again, since this rotation is periodic, we can find such $n$ and $m$ for any $N>0$.  Thus, no limit as $n\rightarrow\infty$ can exist for $K(P_{n,c})$.  

Now suppose $\theta$ is irrational. For any sufficiently small $\epsilon > 0$ let $N > 0$ be given by Corollay \ref{trap2}.  Since the values $\cos(2\pi (n-1) \theta)$ are equidistributed in $[-1,1]$ according to $\cos_{\ast}(\mbox{Leb})$ (where Leb is the Lebesgue measure on the circle) \cite{KATOK}, there will be arbitrarily large values of $m,n > N$ such that $\cos(2\pi (n-1) \theta)  < -1/2 -\epsilon$ and $\cos(2\pi(m-1) \theta) > -1/2 + \epsilon$.  In this case $K_{n,c}$ contains the disc $\mathbb D_{1-\epsilon}$ while, $\mathbb D_{1-\epsilon}$ is contained in the complement of $K_{m,c}$.  Thus, no limit as $n\rightarrow\infty$ can exist for $K(P_{n,c})$.  


Having established the claim in Theorem \ref{main} that no limit exists, we move on to prove the claim that if $\theta$ is rational, $\theta\neq0$, and $\theta\neq\frac{3q\pm1}{3(6p-1)}$, then there are subsequences $a_k$ and $b_k$ partitioning $\{n\in\mathbb N\colon n\geq N\}$ such that
\[\displaystyle\lim_{k\rightarrow\infty}K(P_{a_k,c})\ =\ S^1\mbox{\hspace{.25in}and\hspace{.2in}}
\displaystyle\lim_{k\rightarrow\infty}K(P_{b_k,c})\ =\ \overline{\mathbb D}.\]
We know from Proposition \ref{fixed} that  $|P_{n,c}(c)|\neq1$ for any positive integer $n$.  Thus, for any $\epsilon>0$, we can use Proposition \ref{trap} to find an $N\in\mathbb N$ and construct subsequences
\begin{eqnarray*}
A_{\epsilon}&=&\{n\in\mathbb Z_{+}\colon|P_{n,c}(c)|<1-\epsilon\}\mbox{ and}\\
B_{\epsilon}&=&\{n\in\mathbb Z_{+}\colon|P_{n,c}(c)|>1+\epsilon\}
\end{eqnarray*}
such that for any $n\geq N$,  
\begin{enumerate}
\item if $n\in A_{\epsilon}$, then $K(P_{n,c})$ is full and connected, and 
\item if $n\in B_{\epsilon}$, then $K(P_{n,c})=J(P_{nc})$ is totally disconnected.
\end{enumerate}
Moreover, as $\epsilon\rightarrow0$, these two sets partition $\mathbb N$.  

With the structure of $K(P_{n,c})$ consistent in each of the sets $A_{\epsilon}$ and $B_{\epsilon}$, the remainder of the proof very closely follows the proof of Theorem 1.2 in \cite{boyd}.

Let $\epsilon>0$ and $a_k$ the subsequence of $n\in A_{\epsilon}$.  Then $|P_{a_k,c}(c)|<1-\epsilon$, so by Proposition \ref{formula}, there is an $N_1$ such that for any $a_k\geq N_1$, we have $\mathbb D_{1-\epsilon}\subseteq K(P_{a_k,c})$.  By Lemma \ref{bound}, there is an $N_2\geq N_1$ such that for any $a_k\geq N_2$, we have $K(P_{a_k,c})\subseteq\mathbb D_{1+\epsilon}$.  Thus, for any $z\in K(P_{a_k,c})$,
\begin{equation*}
d(z,\overline{\mathbb D})=\inf_{w\in\overline{\mathbb D}}|z-w|<\epsilon.
\end{equation*}
Now let $w\in\overline{\mathbb D}$.  Since $\mathbb D_{1-\epsilon}\subseteq K(P_{a_k,c})\subseteq\mathbb D_{1+\epsilon}$, we have
\begin{equation*}
d(w,K(P_{a_k,c}))=\inf_{z\in K(P_{a_k,c})}|z-w|<\epsilon.
\end{equation*}
If follows that
\begin{eqnarray*}
d_{\mathcal H}(K(P_{a_k,c}),\overline{\mathbb D})&=&\max\left\{
\sup_{z\in K(P_{a_k,c})}d(z,\overline{\mathbb D}),
\sup_{w\in \overline{\mathbb D}}d(w,K(P_{a_k,c}))\right\}<\epsilon.
\end{eqnarray*}
Thus, $\lim_{k\rightarrow\infty}K(P_{a_k,c})=\overline{\mathbb D}$.

Now let $b_k$ be the subsequence of $n\in B_{\epsilon}$.  Again, by Proposition \ref{formula} and Lemma \ref{bound}, there is an $N_1$ such that for any $b_k\geq N_1$, we have $K(P_{n,c})\subset\mathbb A(1-\epsilon/2,1+\epsilon/2)$.  Also, note that $0\notin K(P_{n,c})$, so $K(P_{nc})$ is totally disconnected and $J(P_{n,c})=K(P_{n,c})$. Then for any $z\in J(P_{b_k,c})$, we have
\begin{equation*}
d(z,S^1)=\inf_{s\in S^1}|z-s|<\epsilon.
\end{equation*}
By Lemma \ref{ball}, there is an $N_2\geq N_1$ such that for any $b_k\geq N_2$ and for any $s\in S^1$,
\begin{equation*}
d(s,J(P_{b_k,c}))=\inf_{z\in J(P_{b_k,c})}|z-s|<\epsilon.
\end{equation*}
Thus, it follows that $d_{\mathcal H}(J(P_{b_k,c}),S^1)<\epsilon$ and $\lim_{k\rightarrow\infty}J(P_{b_k,c})=S^1$.
\end{proof}

\bibliographystyle{plain}
\bibliography{main2}



\end{document}